\documentclass[12pt]{amsart}

\usepackage{amsmath}
\usepackage{amssymb}
\usepackage{esvect}
\usepackage{verbatim}
\usepackage{bm}
\usepackage{graphicx}
\usepackage{psfrag}
\usepackage{color}

\usepackage{hyperref}
\hypersetup{colorlinks=true, linkcolor=blue, citecolor=magenta, urlcolor=wine}
\usepackage{url}
\usepackage{algpseudocode}
\usepackage{fancyhdr}
\usepackage{mathtools}
\usepackage{tikz-cd}
\usepackage{xy}
\input xy
\xyoption{all}
\usepackage{stmaryrd}
\usepackage{calrsfs}

\newtheorem*{maintheorem*}{Main Theorem}
\newtheorem{theorem}{Theorem}[section]
\newtheorem*{theorem*}{Main Theorem}

\newtheorem{prop}[theorem]{Proposition}

\newtheorem{lemma}[theorem]{Lemma}
\newtheorem{cor}[theorem]{Corollary}

\theoremstyle{definition}
\newtheorem{definition}[theorem]{Definition}

\newtheorem{example}[theorem]{Example}
\numberwithin{equation}{section}

\newcommand{\pp}{\mathbb{P}}
\newcommand{\Q}{\mathbb{Q}}
\newcommand{\qq}{\mathbb{Q}}

\newcommand{\Z}{\mathbb{Z}}
\newcommand{\N}{\mathbb{N}}
\newcommand{\R}{\mathbb{R}}

\newcommand{\rr}{\mathbb{R}}
\newcommand{\nn}{\mathbb{N}}



\begin{document}
	
	\mbox{}
	\title{Factorizations in evaluation monoids of \\ Laurent semirings}
	\author{Sophie Zhu}

	\date{\today}

	\begin{abstract}
		For $\alpha\in\rr_{>0}$, let $\nn_0[\alpha,\alpha^{-1}]$ be the semiring of real numbers $f(\alpha)$ with all $f(x) \in \nn_0[x,x^{-1}]$, where $\nn_0$ is the set of nonnegative integers and $\nn_0[x,x^{-1}]$ is the semiring of Laurent polynomials with coefficients in $\nn_0$. In this paper, we study various factorization properties of the additive structure of $\nn_0[\alpha, \alpha^{-1}]$. We characterize when $\nn_0[\alpha, \alpha^{-1}]$ is atomic. Then we characterize when $\nn_0[\alpha, \alpha^{-1}]$ satisfies the ascending chain condition on principal ideals in terms of certain well-studied factorization properties. Finally, we characterize when $\nn_0[\alpha, \alpha^{-1}]$ satisfies the unique factorization property and show that, when this is not the case, $\nn_0[\alpha, \alpha^{-1}]$ has infinite elasticity.
	\end{abstract}

	\maketitle

	\bigskip
	\section{Introduction}
	
	The purpose of this paper is to understand the (additive) factorization properties of the commutative semirings $\nn_0[\alpha, \alpha^{-1}]$ for any $\alpha \in \rr_{> 0}$. To be more precise, let $\nn_0[x,x^{-1}]$ denote the set of Laurent polynomials with coefficients in the set of nonnegative integers $\nn_0$. Since $\nn_0[x,x^{-1}]$ is closed under both addition and multiplication, it is a commutative semiring. For each $\alpha \in \rr_{> 0}$, we let $M_\alpha$ denote the additive monoid of the semiring $\nn_0[\alpha, \alpha^{-1}]$, that is,
	\[
	M_\alpha = \{f(\alpha) \mid f(x) \in \nn_0[x,x^{-1}]\}.
	\]
	It is a sub-semiring of the commutative semiring $\R_{\ge 0}$.
	For ease of notation, we shall use $M_\alpha$ in this paper to denote the additive monoid of the semiring $\nn_0[\alpha, \alpha^{-1}]$ for $\alpha\in\rr_{>0}$. Let $M$ be a cancellative and commutative (additive) monoid. A non-invertible element of $M$ is called an atom if it is not the sum of two non-invertible elements, and $M$ is atomic if every non-invertible element is a sum of atoms. 
	It is well-known that every commutative (and cancellative) monoid satisfying the ascending chain condition on principal ideals (ACCP) is atomic (see, for example, \cite[Proposition 1.1]{pC68}). As for integral domains, $M$ is called a unique factorization monoid (UFM) provided that every non-invertible element  can be written as a sum of atoms in an essentially unique way (i.e., up to order and associates). Here, we study the properties of being atomic, satisfying the ACCP, and being a UFM for the additive monoids $M_\alpha$ (with $\alpha \in \rr_{> 0}$), offering various characterizations for each of such properties in terms of atoms and (additive) factorizations.
	\smallskip
	
	Most of the results we establish here are motivated by some of the results in the recent paper~\cite{CG20} by Correa-Morris and Gotti, where the authors investigated the atomic structure of the additive monoids of the evaluation semirings $\nn_0[\alpha]$ for $\alpha \in \rr_{> 0}$, generalizing some of the results already established by Chapman et al. in~\cite{CGG20} when $\alpha$ is taken in $\qq_{> 0}$. The study of atomicity and factorizations in the setting of commutative semirings has received a great deal of attention in the last few years. For instance, Campanini and Facchini~\cite{CF19} studied the factorization structure of the multiplicative monoid of the semiring $\nn_0[x]$. In addition, Baeth et al.~\cite{BCG21} recently studied the atomic structure of both the additive and the multiplicative monoids of subsemirings of $\rr_{\ge 0}$. Finally, factorizations in certain subsemirings of $\qq_{\ge 0}$ have also been considered in~\cite{ABP21} by Albizu-Campos et al. and in~\cite{BG20} by Baeth and Gotti.
	\smallskip
	
	We begin by introducing the main terminology in Section~\ref{sec:background_1} and outlining the main known results we use later. Then, in Section~\ref{sec:atomicity}, we discuss the atomicity of the monoids $M_\alpha.$ We characterize the monoids $M_\alpha$ that are atomic as well those $M_\alpha$ that are not atomic in Theorem \ref{thm:1atomic} and Proposition \ref{prop:non-atomic characterization}, respectively. In contrast with \cite[Proposition~5.13]{CG20}, the monoid $M_\alpha$ is only finitely generated when $\alpha=1.$ In particular, if $\alpha \neq 1$ and the monoid $M_\alpha$ is atomic, then $M_\alpha$ must contain infinitely many atoms; indeed, we show in Theorem~\ref{thm:1atomic} that the atoms of $M_\alpha$ are precisely the integer powers of~$\alpha$.
	\smallskip
	
	Let $M$ be an atomic (additive) monoid. A factorization of a non-invertible element $x \in M$ is, up to order and associates, a sequence of finitely many atoms (allowing repetitions) with sum $x$, and the number of atoms in such a sequence (counting repetitions) is called the the length of the factorization. A non-invertible element in $M$ may have distinct factorizations (even infinitely many). For a non-invertible element $x \in M$, we let $\mathsf{Z}(x)$ and $\mathsf{L}(x)$ denote the set of factorizations and factorization lengths of $x$, respectively. Following Anderson et al.~\cite{AAZ90} and Halter-Koch~\cite{fHK92}, we say that the monoid~$M$ is an FFM (resp., a BFM) provided that $\mathsf{Z}(x)$ (resp., $\mathsf{L}(x)$) is finite for all non-invertible $x \in M$. The property of being a BFM was first studied back in 1949 by Neumann \cite{bN66} in connection to the ACCP. Note that every FFM is a BFM. In Section~\ref{sec:ACCP}, we prove that the conditions of satisfying the ACCP, being a BFM, and being an FFM are equivalent for any monoid~$M_\alpha$ (see Theorem \ref{thm:ffm}). In addition, we construct monoids $M_\alpha$ that are FFMs but not UFMs (see Subsection~\ref{sub:ffm not ufm}).
	\smallskip
	
	In Section~\ref{sec:factoriality}, we identify the monoids $M_\alpha$ that are UFMs. Following Zaks~\cite{aZ80}, we say that~$M$ is a half-factorial monoid (HFM) if $\mathsf{L}(x)$ is a singleton for every $x \in M$. The property of being an HFM was first considered by Carlitz~\cite{lC60} in the context of algebraic number theory to characterize rings of integers with class number two. Following Chapman et al.~\cite{CCGS21}, we say that $M$ is called a length-factorial monoid (LFM) if for every $x \in M$,  not two factorizations in $\mathsf{Z}(x)$ have the same length. Additionally, in Section~\ref{sec:factoriality}, we prove that the conditions of being a UFM, an HFM, and an LFM are equivalent for any monoid $M_\alpha$. 
	\smallskip
	
	It is not hard to argue that classes satisfying the atomic properties we have just defined are somehow nested, as indicated by the following chain of implications in Diagram~\eqref{eq:atomic chain}. In Section~\ref{sec:factoriality}, we produce a diagram (Diagram~\eqref{eq:refined_chain}) specialized for the class of all monoids $M_\alpha$ that refines Diagram~\eqref{eq:atomic chain}.
	
	\begin{equation} \label{eq:atomic chain}
		\textbf{UFM} \ \Rightarrow \ [\textbf{FFM, HFM}] \ \Rightarrow \ \textbf{BFM} \ \Rightarrow \ \textbf{ACCP} \ \Rightarrow \ \textbf{atomicity}
	\end{equation}
	\smallskip
	
	The elasticity of a monoid is an arithmetic statistic that measures how much a monoid deviates from being an HFM. The elasticity was first considered by Steffan~\cite{jlS86} and Valenza~\cite{rV90} back in the eighties to understand how far from being a UFD is a Dedekind domain or a ring of integers, respectively. Since then the elasticity has become probably the most studied arithmetic invariant to measure non-uniqueness of factorizations (see~\cite{qZ19} by Zhong, and references therein). We conclude this paper with showing that $M_\alpha$ has infinite elasticity when it is not an HFM (see Proposition \ref{prop:elasticity}), which means that either $M_\alpha$ is an HFM or it is as far from being an HFM as a monoid can possibly be.

	\bigskip
	\section{Background}
	\smallskip

	\subsection{General Notation}
	\label{sec:background_1}
	\smallskip
	
	We let $\pp$, $\nn$, and $\nn_0$ denote the set of primes, positive integers, and nonnegative integers, respectively. If $X$ is a subset of $\rr$ and $r$ is a real number, we let $X_{\ge r}$ denote the set $\{x \in X \mid x \ge r\}$. Similarly, we use the notations $X_{> r}, X_{\le r}$, and $X_{< r}$. For a positive rational $q$, the positive integers $a$ and $b$ with $q = a/b$ and $\gcd(a,b) = 1$ are denoted by $\mathsf{n}(q)$ and $\mathsf{d}(q)$, respectively. 
	\smallskip
	
	Given a monic polynomial $f(x)\in\mathbb{Q}[x]$, let $\ell$ be the smallest positive integer such that $\ell \cdot f(x) \in \mathbb{Z}[x]$. Then there exist unique $p(x), q(x) \in \mathbb{N}_0[x]$ such that $\ell f(x) = p(x) - q(x)$ and that $p(x)$ and $q(x)$ share no monomials of the same degree (that is, the greatest common divisor of $p(x)$ and $q(x)$ in the free commutative monoid $(\N_0[x],+)$ is $0$). We call the pair $(p(x), q(x))$ the \emph{minimal pair} of $f(x)$. In addition, if $\alpha$ is a real algebraic number, the \emph{minimal pair of $\alpha$} is defined to be the minimal pair of its minimal polynomial over $\Q$.

	\medskip
	\subsection{Monoids}
	\smallskip
	
	A \emph{monoid} is a cancellative and commutative semigroup with an identity element. Monoids here will be written additively, unless we say otherwise. Let $M$ be a monoid. An element $x \in M$ is called \emph{invertible} (or a \emph{unit}) if there exists $y \in M$ such that $x+y = 0$. We tacitly assume that $M$ (and every monoid we deal with here) is \emph{reduced}; that is, its only invertible element is $0$. We set $M^\bullet = M \setminus \{0\}$. For a subset $S$ of $M$, we let $\langle S \rangle$ denote the submonoid of $M$ generated by $S$, i.e., the intersection of all submonoids of $M$ containing $S$. We say that a monoid is \emph{finitely generated} if it can be generated by a finite set. A nonzero element $a \in M$ is called an \emph{atom} if whenever $a = x+y$ for some $x,y \in M$ either $x = 0$ or $y = 0$. It is customary to let $\mathcal{A}(M)$ denote the set consisting of all atoms of $M$, and we do so. If $\mathcal{A}(M)$ is empty, $M$ is said to be \emph{antimatter}. The monoids we are mostly interested in this paper are atomic.
	
	\begin{definition}
		An (additive) monoid is \emph{atomic} if every nonzero element can be written as a sum of atoms.
	\end{definition}
	
	If $I$ is a subset of $M$, then $I$ is called an \emph{ideal} provided that $I + M = I$ (or, equivalently, $I + M \subseteq I$). Every subset of $M$ of the form $x + M$, where $x \in M$, is an ideal and is called a \emph{principal} ideal. The monoid $M$ satisfies the \emph{ascending chain condition on principal ideals} (\emph{ACCP}) if every increasing sequence (under inclusion) of principal ideals of $M$ becomes stationary from one point on. It is well known that every monoid satisfying the ACCP is atomic (see \cite[Proposition~1.1.4]{GH06}). The converse does not hold: for instance, the additive submonoid $\langle (\frac{2}{3})^n \mid n \in \nn \rangle$ of $\qq$ is an atomic monoid that does not satisfy the ACCP \cite[Corollary~4.4]{CGG21}.
	\smallskip
	
	\subsection{Factorizations} Assume now that $M$ is atomic. Let $\mathsf{Z}(M)$ denote the free (commutative) monoid on the set $\mathcal{A}(M)$. For each $x \in M$, we let $\mathsf{Z}(x)$ denote the set of all formal sums $z := a_1 +  \cdots + a_\ell \in \mathsf{Z}(M)$ with $a_1, \dots, a_\ell \in \mathcal{A}(M)$ such that $a_1 + \dots + a_\ell = x$ in $M$. In this case, $\ell$ is called the \emph{length} of $z$ and is denoted by~$|z|$. For each $x \in M$, we set $\mathsf{L}(x) := \{ |z| \mid z \in \mathsf{Z}(x) \}$. The sets $\mathsf{Z}(x)$ and $\mathsf{L}(x)$ play an important role in factorization theory (see~\cite{aG16}). Note that~$M$ is atomic if and only if $\mathsf{Z}(x)$ is nonempty for all $x \in M^\bullet$. 
	\smallskip
	
	The monoid $M$ is called a \emph{bounded factorization monoid} (BFM) if $\mathsf{L}(x)$ is finite for all $x \in M$. Every BFM satisfies the ACCP \cite[Corollary~1.3.3]{GH06}, but the converse does not hold: $\langle 1/p \mid p \in \pp \rangle$ satisfies the ACCP but is not a BFM \cite[Corollary~4.6]{CGG21}. The monoid $M$ is called a \emph{half-factorial monoid} (HFM) if $|\mathsf{L}(x)| = 1$ for all $x \in M^\bullet$. Observe that every HFM is a BFM. The monoid $M$ is called a \emph{finite factorization monoid}) (FFM) if $\mathsf{Z}(x)$ is finite for all $x \in M^\bullet$. Every finitely generated monoid is an FFM \cite[Corollary~3.7]{AG21}. Note that every FFM is a BFM; however, $\{0\} \cup \qq_{\ge 1}$ is a BFM that is not an FFM \cite[Example~4.10]{CGG21}. In addition, one can see that $\langle 2,3 \rangle$ is an FFM that is not an HFM. On the other hand, there are HFMs that are not FFMs; this is the case of the additive monoid $\{(0,0)\} \cup (\Z \times \nn)$ (see \cite[Example~3.9]{AG21}). Finally,~$M$ is called a \emph{unique factorization monoid} (UFM) provided that $|\mathsf{Z}(x)| = 1$ for all $x \in M^\bullet$. Every UFM is, by definition, both an HFM and an FFM. Then we see that each implication in Diagram~\eqref{eq:atomic chain} holds and that such a diagram does not support, in general, any additional implication.

	\bigskip
	\section{Atomicity}
	\label{sec:atomicity}
	
	In this section, we study the atomicity of the additive monoids $M_\alpha$, where $M_\alpha =\nn_0[\alpha,\alpha^{-1}]$ for $\alpha\in\rr_{>0}$. We characterize the monoids $M_\alpha$ that are atomic, and then we give examples of monoids $M_\alpha$ that are atomic but do not satisfy the ACCP. The next theorem, which gives a simple characterization of the monoids $M_\alpha$ that are atomic, also provides an explicit description of the set of atoms of $M_\alpha$. Moreover, it gives a necessary condition for the atomicity of $M_\alpha$ when $\alpha$ is algebraic. For any algebraic number $\alpha$ with minimal polynomial $m(x)\in\Q[x]$, the polynomial $\ell \cdot m(x)$ is a primitive polynomial in $\Z[x]$ for a unique $\ell \in \nn$, so $\ell\cdot m(x)=p(x)-q(x)$ for unique $p(x),q(x)\in\N_0[x]$ that do not share monomials of equal degrees. We call $(p(x),q(x))$ the minimal pair of $\alpha$ (see Section \ref{sec:background_1}).
	
	\begin{theorem}\label{thm:1atomic}
		For each $\alpha \in \mathbb{R}_{>0},$ the following statements are equivalent.
		\begin{enumerate}
			\item[(a)] $1 \in \mathcal{A}(M_\alpha)$.
			\smallskip
			
			\item[(b)] $\mathcal{A}(M_\alpha) = \{\alpha^n \mid n \in \mathbb{Z}\}$.
			\smallskip
			
			\item[(c)] $M_\alpha$ is atomic.
		\end{enumerate}
		Suppose that $\alpha \in \mathbb{R}_{>0} \setminus \{1\}$ is an algebraic number. If $M_\alpha$ is atomic, then neither of the two components in the minimal pair of $\alpha$ is a  monic monomial.
	\end{theorem}
	
	\begin{proof}
		(a) $\Rightarrow$ (b): Suppose that $\mathcal{A}(M_\alpha) \neq \{\alpha^n \mid n \in \mathbb{Z}\}$. Then there exists $n \in \mathbb{Z}$ such that $\alpha^n \not\in \mathcal{A}(M_\alpha)$ and, therefore, there exists a finite set $S \subset \mathbb{Z}$ such that $\alpha^n = \sum_{i \in S} c_i\alpha^i$ for some coefficients $c_i \in \mathbb{N}$ for each $i \in S$ such that $\sum_{i\in S} c_i \geq 2$. Dividing by $\alpha^n$ gives $1 = \sum_{i\in S} c_i\alpha^{i-n}.$ Thus, $1 \not\in \mathcal{A}(M_\alpha),$ as desired. 
		\smallskip
		
		(b) $\Rightarrow$ (c): This holds by the definition of $M_\alpha$.
		\smallskip
		
		(c) $\Rightarrow$ (a): Suppose $1\not\in \mathcal{A}(M_\alpha).$  Then there exists a finite set $S \subset \mathbb{Z}$ and coefficients $c_i \in \mathbb{N}$ for each $i \in S$ such that $\sum_{i\in S} c_i \geq 2$ and $1 = \sum_{i \in S} c_i\alpha^i.$ For each $k \in \Z$, we can multiplying both sides of $1 = \sum_{i \in S} c_i\alpha^i$ by $\alpha^k$ to obtain the equality $\alpha^k = \sum_{i\in S} c_i\alpha^{i+k}.$ Thus, $\alpha^k$ is not an atom for any $k \in \mathbb{Z}$, which implies that $M_\alpha$ has no atoms and, therefore, that it is not atomic.
		\smallskip
		
		Assume now that $\alpha$ is a positive algebraic real number such that $\alpha \neq 1$. Let $m(x)$ and $(p(x), q(x))$ be the minimal polynomial and the minimal pair of $\alpha$, respectively. Suppose, by way of contradiction, that either $p(x)$ or $q(x)$ is a monic monomial. We can say, without loss of generality, that $q(x) = x^n$ for some $n \in \mathbb{N}_0$. Thus, $p(\alpha) - \alpha^n = p(\alpha) - q(\alpha) = \ell m(\alpha) = 0$ for some $\ell \in \N,$ so $p(\alpha) = \alpha^n$. Because $\alpha \neq 1$, we see that $p(x)$ must be the sum of at least two nonzero monomials (not necessarily distinct). Consequently, $\alpha^n \notin \mathcal{A}(M_\alpha)$. Therefore, $M_\alpha$ is not atomic in light of the characterizations established above, which yields the desired contradiction.
	\end{proof}
	
	It is worth mentioning that, as a direct consequence of Theorem~\ref{thm:1atomic}, one obtains that every monoid $M_\alpha$ satisfies $|\mathcal{A}(M_\alpha)| \in \{0, \infty\}$ and also that $M_\alpha$ is either atomic or antimatter. In addition, when $\alpha$ is transcendental, $M_\alpha$ is atomic, as we now illustrate.
	
	\begin{cor}
		If $\alpha \in \rr_{> 0}$ is transcendental, then $M_\alpha$ is atomic.
	\end{cor}
	
	\begin{proof}
		Suppose that $1 = \sum_{i \in S} c_i \alpha^i$ for a finite set $S \subseteq \Z$ and coefficients $c_i \in \nn_0$ for every $i \in S$. Then $\alpha$ would be a root of the polynomial $f(x) := x^m - \sum_{i \in S} c_i x^{i+m} \in \Z[x]$, where $m = -\min (\{0\} \cup S)$. Since $\alpha$ is transcendental, $f(x)$ is the zero polynomial and, therefore, $S = \{0\}$ and $c_0 = 1$. Hence $1 \in \mathcal{A}(M_\alpha)$, and $M_\alpha$ is atomic by Theorem~\ref{thm:1atomic}.
	\end{proof}
	
	It is worth emphasizing that the necessary condition in Theorem~\ref{thm:1atomic} is not sufficient; this is illustrated in the following example.
	
	\begin{example}  \label{ex:non-atomic monoid M_alpha}
		Consider the monic polynomial $m(x) = x^3 - 2x^2 + 3x - 7$. Because $m(x)$ has no integer roots, it follows from Gauss's lemma that $m(x)$ is irreducible in $\Q[x]$. On the other hand, $m(2) = -1$ and $m(3) = 11$, the polynomial $m(x)$ has a positive root $\alpha$ in the interval $(2,3)$. Consider the monoid $M_\alpha$. As $m(x)(x + 2) = x^4 - x^2 - x - 14$, we see that $\alpha$ is a root of the polynomial $x^4 - x^2 - x - 14$, so $\alpha^4 = \alpha^2 + \alpha + 14$. Hence $\alpha$ is not an atom of $M_\alpha$, and it follows from the characterization in Theorem~\ref{thm:1atomic} that $M_\alpha$ is not atomic. However, none of the polynomials in the minimal pair $(x^3 + 3x, 2x^2 + 7)$ of $\alpha$ are monic monomials. Therefore we conclude that the necessary condition in Theorem~\ref{thm:1atomic} is not sufficient.
	\end{example}
	
	If $\alpha = 1$, then $M_\alpha = \nn_0$, which is atomic. On the other hand, if $\alpha \in \nn_{\ge 2}$ (or if $\alpha = 1/n$ for some $n \in \nn_{\ge 2}$), then $1$ is the sum of $\alpha$ copies of $\alpha^{-1}$ (resp., $\alpha^{-1}$ copies of $\alpha$) and, therefore $1 \notin \mathcal{A}(M_\alpha)$, and so Theorem~\ref{thm:1atomic} ensures that $M_\alpha$ is not atomic. In addition, we have exhibited in Example~\ref{ex:non-atomic monoid M_alpha} a monoid $M_\alpha$ that is not atomic for some $\alpha \in \rr_{> 0} \setminus \qq$. We now characterize the monoids $M_\alpha$ that are not atomic.
	
	\begin{prop} \label{prop:non-atomic characterization}
		For $\alpha \in \rr_{> 0}$ with $\alpha \neq 1$, the following statements are equivalent.
		\begin{enumerate}
			\item[(a)] $M_\alpha$ is not atomic.
			\smallskip
			
			\item[(b)] $(\nn_0[\alpha],+)$ is antimatter or finitely generated.
		\end{enumerate}
	\end{prop}
	
	\begin{proof}
		(a) $\Rightarrow$ (b): Suppose that $M_\alpha$ is not atomic. Then $\alpha$ is algebraic as otherwise $M_\alpha$ would be a free commutative monoid, which is atomic. We consider the following two cases.
		\smallskip
		
		\textit{Case 1:} $\alpha < 1$. Since $M_\alpha$ is not atomic, $1 \notin \mathcal{A}(M_\alpha)$ by Theorem~\ref{thm:1atomic}, so we can write $1 = \sum_{i=1}^n c_i \alpha^i$ for some $n \in \nn$ and $c_1, \dots, c_n \in \nn_0$ (here, we use that $\alpha  < 1$). Then $1$ is not an atom of the additive monoid $\nn_0[\alpha]$, and it follows from \cite[Theorem~4.1]{CG20} that $\nn_0[\alpha]$ is antimatter.
		\smallskip
		
		\textit{Case 2:} $\alpha > 1$. Since $0$ is not a limit point of $\nn_0[\alpha]^\bullet$ (because $\alpha > 1$), it follows from \cite[Proposition~4.5]{fG19} that $\nn_0[\alpha]$ is atomic. As in the case already considered, the fact that $M_\alpha$ is not atomic allows us to write $1 = \sum_{i=1}^n c_i  \alpha^{-i}$ for some $n \in \nn$ and $c_1, \dots, c_n \in \nn_0$ (here, we use that $\alpha > 1$). Therefore $\alpha^n = \sum_{i=1}^n c_i \alpha^{n-i}$. As $\nn_0[\alpha]$ is an atomic monoid, the inclusion $\mathcal{A}(\nn_0[\alpha]) \subseteq \{\alpha^k \mid k \in \{0, \dots,  n-1 \} \}$ holds by \cite[Theorem~4.1]{CG20}. Thus, $\nn_0[\alpha]$ is finitely generated.
		\smallskip
		
		(b) $\Rightarrow$ (a): Note that $\alpha$ is algebraic, for otherwise, $\N_0[\alpha]$ would be a free commutative monoid on a countable basis, which is neither antimatter nor finitely generated. Suppose first that the additive monoid $\nn_0[\alpha]$ is antimatter. Since the set $\{\alpha^n \mid n \in \nn_0\}$ generates $\nn_0[\alpha]$, the equality $1 = \sum_{i=1}^k c_i \alpha^i$ holds for some $k \in \nn$ and $c_1, \dots, c_k \in \nn_0$. Hence $1 \notin \mathcal{A}(M_\alpha)$, and so $M_\alpha$ is not atomic by Theorem~\ref{thm:1atomic}.
		\smallskip
		
		
		Finally, suppose that the additive monoid $\N_0[\alpha]$ is finitely generated. Then $\nn_0[\alpha]$ is atomic by \cite[Proposition~2.7.8]{GH06}, and it follows from \cite[Theorem~4.1]{CG20} that
		\[
		\mathcal{A}(\nn_0[\alpha]) = \{\alpha^k \mid k \in \{ 0, \dots, n-1 \} \}
		\]
		for some $n \in \nn$. Since $\alpha \neq 1$, we see that $n \ge 2$. Then $\alpha^n = \sum_{k=0}^{n-1} c_k \alpha^k$ for some $c_0, \dots, c_{n-1} \in \nn_0$, which means that $1 = \sum_{k=0}^{n-1} c_k \alpha^{k-n}$. Hence $1 \notin \mathcal{A}(M_\alpha)$, and it follows from Theorem~\ref{thm:1atomic} that $M_\alpha$ is not atomic.
	\end{proof}
	
	We conclude this section with examples of monoids $M_\alpha$ that are atomic but do not satisfy the ACCP.
	
	\begin{example} \label{ex:atomic monoids without the ACCP}
		Take $a,b \in \nn$ with $\gcd(a,b) = 1$ such that $1 < a < b$, and set $\alpha = a/b$. It follows from \cite[Proposition~3.5]{fG18} that the monoid $M_\alpha$ is atomic. On the other hand, we claim that $M_\alpha$ does not satisfy the ACCP. By \cite[Corollary 4.4]{CGG21}, there is an ascending chain $(x_n + \nn_0[\alpha])_{n \in \nn}$ of principal ideals of the monoid $(\N_0[\alpha], +)$ that does not stabilize. From the fact that $M_\alpha$ is a reduced monoid having $(\nn_0[\alpha],+)$ as a submonoid, we can deduce that the chain of principal ideals $(x_n + M_\alpha)_{n \in \nn}$ of $M_\alpha$ cannot stabilize in $M_\alpha,$ showing that $M_\alpha$ does not satisfy the ACCP.
	\end{example}

	\bigskip
	\section{The Ascending Chain Condition on Principal Ideals}
	\label{sec:ACCP}

	We have just seen in the previous section that satisfying the ACCP is a stronger condition than being atomic when restricted to the class consisting of the monoids $M_\alpha$. In this section, we provide two necessary conditions for a monoid $M_\alpha$ to satisfy the ACCP, and then we establish two factorization-theoretical characterizations: satisfying the ACCP is equivalent to both the bounded factorization property and the finite factorization property if one restricts attention to the class consisting of all monoids $M_\alpha$. We conclude this section by constructing monoids $M_\alpha$ satisfying the ACCP that are not UFMs.
	
	\begin{prop} \label{prop:accp}
		Let $\alpha \in (0,1)$ be an algebraic number with minimal pair $(p(x),q(x))$. If $M_\alpha$ satisfies the ACCP, then $p(x) - Q(x)q(x) \notin \N_0[x,x^{-1}]$ for any nonzero Laurent polynomial $Q(x) \in \N_0[x,x^{-1}]$.
	\end{prop}
	
	\begin{proof}
		Suppose, for the sake of contradiction, that there exists a nonzero Laurent polynomial $Q(x) \in \N_0[x,x^{-1}]$ such that $r(x) := p(x) - Q(x)q(x) \in \N_0[x,x^{-1}]$. Now consider the sequences $(a_n)_{n \in \N}$ and $(b_n)_{n \in \N}$ defined by
		\[
		a_n = Q(\alpha)^n q(\alpha) \quad  \text{ and } \quad b_n := Q(\alpha)^n r(\alpha),
		\]
		respectively, for every $n \in \N$. Observe that the terms of both $(a_n)_{n \in \N}$ and $(b_n)_{n \in \N}$ are nonzero elements in $M_\alpha$. On the other hand,
		\[
		a_n = Q(\alpha)^n q(\alpha) = Q(\alpha)^n p(\alpha) = Q(\alpha)^{n+1} q(\alpha) + Q(\alpha)^{n} r(\alpha) = a_{n+1} + b_{n}
		\] 
		for every $n \in \N$. Therefore $(a_n + M_\alpha)_{n \in \N}$ is an ascending chain of principal ideals of $M_\alpha$. Since $a_n - a_{n+1} = b_n > 0$ for every $n \in \N$, the chain of ideals $(a_n + M_\alpha)_{n \in \N}$ does not stabilize, contradicting that $M_\alpha$ satisfies the ACCP.
	\end{proof}
	
	In Example~\ref{ex:atomic monoids without the ACCP}, we saw that $M_\alpha$ is an atomic monoid that does not satisfy the ACCP for most choices of $q \in \qq_{> 0}$. However, there are also some examples of irrational algebraic real numbers $\alpha$ such that $M_\alpha$ is atomic but does not satisfy the ACCP, and we can identify some of them using Proposition~\ref{prop:accp}.
	
	\begin{example} \label{ex:quadratic atomic monoids without the ACCP}
		Take $a,b \in \N$ such that $\gcd(a,b) = 1$ and $1 < a < b$. Assume, in addition, that $a$ and $b$ are not perfect squares, and then set $\alpha := \sqrt{a/b}$. Then $\alpha$ is a non-rational algebraic number with minimal polynomial $m(x) := x^2 - a/b$. Suppose, by way of contradiction, that $M_\alpha$ is not atomic. By Theorem~\ref{thm:1atomic}, we can take $c_1, \dots, c_n \in \N_0$ such that $1 = c_1 \alpha + \cdots + c_n \alpha^n$. Since $\alpha$ is a root of the polynomial $f(x) := c_n x^n + \cdots + c_1 x - 1 \in \Z[x]$, there exists a polynomial $g(x) \in \Q[x]$ such that $f(x) = m(x)g(x)$. By Gauss's lemma, there exists $q \in \qq_{> 0}$ such that $m'(x) := qm(x) \in \Z[x]$ and $g'(x) := q^{-1}g(x) \in \Z[x]$. Since $qm(x)$ has integer coefficients, $q \in b\N$. Therefore $a \mid m'(0)$, so $a \mid m'(0) g'(0) = f(0) = 1$, a contradiction. Thus, $M_\alpha$ is atomic. Let us argue now that $M_\alpha$ does not satisfy the ACCP. Since $\alpha$ has minimal pair $(p(x), q(x)) := (b x^2,a)$, for $Q(x) := x^2$ we see that $p(x) - Q(x) q(x) = (b-a)x^2$, which belongs to $\nn_0[x,x^{-1}]$. Hence $M_\alpha$ does not satisfy the necessary condition in Proposition~\ref{prop:accp}, and so it does not satisfy the ACCP.
	\end{example}

	\medskip
	\subsection{The Bounded and Finite Factorization Properties}
	
	In this subsection we prove that in the context of the monoids $M_{\alpha}$, satisfying the ACCP, being a BFM, and being an FFM are equivalent properties. Recall that a monoid $M$ is a BFM (resp., an FFM) provided that $\mathsf{L}(x)$ (resp., $\mathsf{Z}(x)$) is finite for all $x \in M^\bullet$. We proceed to establish the main result of this section.
	
	\begin{theorem}\label{thm:ffm}
		For $\alpha \in \mathbb{R}_{> 0},$ the following statements are equivalent. 
		\begin{enumerate}
			\item[(a)] $M_\alpha$ is an FFM.
			\smallskip
			
			\item[(b)] $M_\alpha$ is a BFM.
			\smallskip
			
			\item[(c)] $M_\alpha$ satisfies the ACCP.
		\end{enumerate}
	\end{theorem}
	
	\begin{proof}
		(a) $\Rightarrow$ (b): This follows from the definitions of a BFM and an FFM.
		\smallskip
		
		(b) $\Rightarrow$ (c): This is a special case of \cite[Corollary 1.3.3]{GH06}.
		\smallskip
		
		(c) $\Rightarrow$ (a): Suppose that the monoid $M_\alpha$ satisfies the ACCP. If $\alpha$ is transcendental, then $M_\alpha$ is a free commutative monoid, and thus an FFM. We assume, therefore, that $\alpha$ is algebraic.
		\smallskip
		
		Suppose, by way of contradiction, that $M_\alpha$ is not an FFM. Then $\alpha \neq 1$ and, after replacing $\alpha$ by $\alpha^{-1}$ if necessary, we can assume that $\alpha > 1$. Since $M_\alpha$ is not an FFM, we can choose $\beta \in M_\alpha$ such that $|\mathsf{Z}_{M_\alpha}(\beta)| = \infty$. Because $\alpha > 1$, there exists $N \in \nn$ such that $\alpha^n \nmid_{M_\alpha} \beta$ for any $n \in \Z$ with $n > N$. As $M_\alpha$ is atomic, $\mathcal{A}(M_\alpha) = \{\alpha^n \mid n \in \Z\}$ by Theorem~\ref{thm:1atomic}. Consequently, there is a bijection $\mathsf{Z}_{M_\alpha}(\beta) \to \mathsf{Z}_{M_\alpha}(\beta/\alpha^N)$ given by multiplication by $\alpha^{-N}$. In addition, $\beta/\alpha^N$ is not divisible by any positive power of $\alpha$ in $M_\alpha$. Then after replacing $\beta$ by $\beta/\alpha^N,$ we can further assume that $\alpha^k \mid_{M_\alpha} \beta$ implies that $k \le 0$. Since $M_{\alpha^{-1}} = M_\alpha$ is atomic, it follows from Proposition~\ref{prop:non-atomic characterization} that the additive monoid $\N_0[\alpha^{-1}]$ is neither antimatter nor finitely generated. Hence, \cite[Theorem~4.1]{CG20} guarantees that $\mathcal{A}(\nn_0[\alpha^{-1}]) = \{\alpha^{-k} \mid k \in \nn_0\}$. As a result, the fact that $\alpha^k \nmid_{M_\alpha} \beta$ for any $k \in \nn$ ensures that $\mathsf{Z}_{M_\alpha}(\beta) = \mathsf{Z}_{\nn_0[\alpha^{-1}]}(\beta)$ and, therefore, that $|\mathsf{Z}_{\nn_0[\alpha^{-1}]}(\beta)| = \infty$. Thus, $\nn_0[\alpha^{-1}]$ is not an FFM. Now it follows from \cite[Theorem~4.11]{CG20} that $\nn_0[\alpha^{-1}]$ does not satisfy the ACCP. However, this is a contradiction to the fact that $\nn_0[\alpha^{-1}]$ is a submonoid of the reduced monoid $M_\alpha$, which satisfies the ACCP. Hence, $M_\alpha$ is an FFM.
	\end{proof}

	\medskip
	\subsection{A Class of FFMs that are not UFMs} 
	\label{sub:ffm not ufm}
	
	We have exhibited in Examples~\ref{ex:atomic monoids without the ACCP} and~\ref{ex:quadratic atomic monoids without the ACCP} some atomic monoids $M_\alpha$ that do not satisfy the ACCP. However, the only examples we have so far of monoids $M_\alpha$ satisfying the ACCP (or, equivalently, being FFMs) are the trivial cases, namely, those corresponding to $\alpha = 1$ and $\alpha$ transcendental. Our purpose in this subsection is to produce monoids $M_{\alpha}$ that are FFMs for some algebraic~$\alpha$ different from $1$. This will yield monoids $M_\alpha$ that are FFMs but not UFMs.
	\smallskip
	
	To do so, let $\alpha_1, \alpha_2 \in \R$ be distinct roots of an irreducible quadratic polynomial in $\Q[x]$, and set $M := M_{\alpha_1}$ and $K := \mathbb{Q}(\alpha_1)$. Then $K$ is a real quadratic field extension of~$\Q$ that contains the monoid $M$. In addition, let $T \colon \qq(\alpha_1) \to \rr^2$ be the $\Q$-linear map induced by the assignments $1 \mapsto (1,1)$ and $\alpha_1 \mapsto (\alpha_1,\alpha_2)$, and set $M' = T(M)$. Let $T_M \colon M \to M'$ be the map obtained by restricting the domain and codomain of~$T$ to $M$ and $M'$, respectively. We use the notation introduced in this paragraph throughout the rest of this section.
	
	\begin{lemma} \label{lem:monoid M'}
		The following statements hold. 
		\begin{enumerate}
			\item $T$ is an injective $\Q$-algebra homomorphism. 
			\smallskip
			
			\item $T_M$ is a monoid isomorphism.
			\smallskip
			
			\item $M' = \big\{ \sum_{i \in I} c_i (\alpha_1^i, \alpha_2^i) \mid c_i \in \N_0, I \subseteq \Z, |I| < \infty \big\}$.
		\end{enumerate}
	\end{lemma}
	
	\begin{proof}
		(1) Since $T$ is a $\Q$-linear map, the equalities $T(x+y) = T(x)+T(y)$ and $T(qx) = q T(x)$ hold for all $x,y \in \qq(\alpha_1)$ and $q \in \Q$. For each $i \in \{1,2\}$, we let~$\sigma_i$ denote the $\qq$-algebra homomorphism $\Q(\alpha_1) \to \R$ induced by the assignment  $\alpha_1 \mapsto \alpha_i$. Then for each $x \in \Q(\alpha_1)$, we can verify that $T(x) = (\sigma_1(x),\sigma_2(x)).$ Therefore, for all $x, y \in \Q(\alpha_1)$,
		\[
		T(xy) = (\sigma_1(xy), \sigma_2(xy)) = (\sigma_1(x)\sigma_1(y), \sigma_2(x)\sigma_2(y)) = T(x) T(y).
		\]
		Hence $T$ is a $\Q$-algebra homomorphism. Note that $T(\alpha_1^{-1}) = T(\alpha_1)^{-1} = (\alpha_1^{-1},\alpha_2^{-1})$. Now if $x \in \ker T$ for some $x \in \qq(\alpha_1)$, then $(\sigma_1(x), \sigma_2(x)) = (0,0)$, and so the fact that $\sigma_1$ is the inclusion map ensures that $x=0$. Thus, $T$ is an injective $\qq$-algebra homomorphism.
		\smallskip
		
		(2) Since $T$ is injective, it is also injective when restricted to $M \subseteq \Q(\alpha_1)$. Moreover, because $M'$ is the image of $M$ under $T$, the map $T_M \colon M \to M'$ is a bijection. In addition, the linearity of $T$ over $\qq$ immediately implies that $T_M$ is a monoid homomorphism, making it a monoid isomorphism from $M$ onto $M'$.
		\smallskip
		
		(3) Finally, let $x$ be an arbitrary element in $M$. Then $x = \sum_{i \in I} c_i\alpha_1^i$ for a finite subset $I$ of $\Z$ and coefficients $c_i \in \N_0$. Because $T$ is a $\Q$-algebra homomorphism by part~(1), we see that
		\[
		T(x) = \sum_{i \in I} c_i T(\alpha_1)^i = \sum_{i \in I} c_i (\alpha_1,\alpha_2)^i= \sum_{i \in I} c_i (\alpha_1^i, \alpha_2^i).
		\]
		Therefore  $M' \subseteq \big\{ \sum_{i \in I} c_i (\alpha_1^i, \alpha_2^i) \mid c_i \in \N_0, I \subseteq \Z, |I| < \infty \big\}$. The reverse implication follows immediately as $T$ is a $\qq$-algebra homomorphism and $M$ is a monoid containing $\{\alpha_1^i \mid i \in \Z\}$.
	\end{proof}
	
	In order to establish the main result of this section, we need the following two lemmas.
	
	\begin{lemma} \label{lem:decreasing}
		Let $M$ be an additive submonoid of $\rr^2_{\ge 0}$. If $v, w \in M$ with $v = (v_1, v_2)$ satisfy $v + M \subseteq w + M$, then $w \in [0,v_1] \times [0,v_2]$. 
	\end{lemma}
	
	\begin{proof}
		Since $v + M \subseteq w + M$, we see that $w$ divides $v$ in $M$ and, therefore, we can write $v = w + d$, where $d = (d_1, d_2) \in M \subseteq \rr_{\ge0}^2$. Then $w = (v_1 - d_1, v_2 - d_2)$ belongs to $[0, v_1] \times [0, v_2]$.
	\end{proof}
	
	For the rest of this section, we further assume that $0 < \alpha_1 < 1 < \alpha_2$. We observe that, in light of part~(3) of Lemma~\ref{lem:monoid M'}, the inclusion $M' \subseteq \{(0,0)\} \cup \rr_{> 0} \times \rr_{> 0}$ holds. 
	
	\begin{lemma} \label{lem:finite}
		If $(v_1,v_2) \in M'$, then the set $M' \cap ([0,v_1] \times [0, v_2])$ is finite. 
	\end{lemma}
	
	\begin{proof}
		Set $v := (v_1,v_2)$ and $S_v := M' \cap ([0,v_1] \times [0, v_2])$. If $v = (0,0)$, then $S_v$ is a singleton and thus finite. Now we assume that $v \neq (0,0)$. Note that since $\alpha_1^{-1} > 1$ and $\alpha_2 > 1,$ the sequences $(\alpha_1^{-n})_{n \in \N_0}$ and $(\alpha_2^n)_{n \in \N_0}$ both increase to infinity and, as a result, the nonempty set
		\begin{equation*}
			N := \{ n \in \Z \mid \alpha_1^n \leq v_1 \text{ and } \alpha_2^n \leq v_2\}
		\end{equation*}
		is bounded. Let $m$ be the maximum of $N$. Take a nonzero $s \in S_v$. Since $T$ is injective, there exists a unique $\alpha \in M$ such that $s = T(\alpha)$. Write
		\begin{equation} \label{eq:representation of alpha}
			\alpha = \sum_{i=0}^m q_i\alpha_1^{-i} + \sum_{i=0}^m p_i \alpha_1^i \in M^\bullet,
		\end{equation}
		where $q_0, \dots, q_m$ and $p_0, \dots, p_m$ are nonnegative integers. As a result, we see that
		\begin{align*}
			s  &= \sum_{i=0}^{m} q_i T(\alpha_1^{-i}) + \sum_{i=0}^{m} p_i T(\alpha_1^i) \\
			&= \sum_{i=0}^{m} q_i (\alpha_1^{-i}, \alpha_2^{-i}) + \sum_{i=0}^{m} p_i (\alpha_1^i,\alpha_2^i) \\
			&= \bigg( \sum_{i=0}^{m}( q_i \alpha_1^{-i} + p_i \alpha_1^i), \, \sum_{i=0}^{m} (q_i \alpha_2^{-i} + p_i \alpha_2^i )\bigg).
		\end{align*}
		Because $\alpha_2 > 1$, after looking at the second coordinate of $s$, we infer that $p_i \le p_i \alpha_2^i \le v_2$ for every $i \in \{0, \dots, m\}$. Hence, there are at most $(v_2 + 1)^{m+1}$ many possible $(m+1)$-tuples $(p_0, p_1, \dots, p_m)$ to choose for the respective coefficients of  $\alpha_1^0, \dots, \alpha_1^m$ for a representation of $\alpha$ as in~\eqref{eq:representation of alpha}. Symmetrically, since $\alpha_1^{-1} > 1,$ there are finitely many possible $(m+1)$-tuples $(q_0, q_1, \dots, q_m)$ one can choose to express $\alpha$ as in~\eqref{eq:representation of alpha}. Consequently, the set $T_M^{-1}(S_v)$ is finite, which implies that $S_v$ is also finite.
	\end{proof}
	
	We are in a position to prove that $M$ is an FFM.
	
	\begin{theorem} \label{thm:FFM evaluation monoids}
		Suppose that $\alpha_1$ and $\alpha_2$ are the roots of an irreducible quadratic polynomial in $\Q[x]$ such that $0 < \alpha_1 < 1< \alpha_2$. Then $M_{\alpha_1}$ is an FFM and, therefore, satisfies the ACCP.
	\end{theorem}
	
	\begin{proof}
		Define $T \colon \qq(\alpha_1) \to \rr^2$ and $M'$ as before. Let $v = (v_1, v_2)$ be a nonzero element in $M'$. It follows from Lemma~\ref{lem:finite} that $S_v := M' \cap ([0, v_1] \times [0, v_2])$ is a finite set. On the other hand, it follows from Lemma~\ref{lem:decreasing} that every divisor of $v$ in $M'$ belongs to~$S_v$. Therefore, $v$ has only finitely many divisors in $M'$ and, as a result, $M'$ is an FFM by virtue of \cite[Theorem 2]{fHK92}. Since $M$ is isomorphic to $M'$ and being an FFM is an algebraic property, we conclude that $M$ is an FFM, whence it satisfies the ACCP.
	\end{proof}
	
	There are monoids $M_\alpha$ that are FFMs but not UFMs. The following example illustrates this observation.
	
	\begin{example} \label{ex:FFM that is not a UFM}
		Consider the polynomial $p(x) := x^2 - 2x + \frac{1}{2} \in \qq[x]$. Since the roots of $p(x)$ are $\alpha := 1 - \frac{\sqrt{2}}2$ and $\beta := 1 + \frac{\sqrt{2}}2$, it is an irreducible polynomial. In light of Theorem~\ref{thm:FFM evaluation monoids}, the chain of inequalities $0 < \alpha < 1 < \beta$ guarantees that the additive monoid $M_\alpha$ is an FFM. However, $M_\alpha$ is not a UFM: indeed, since $1,\alpha, \alpha^2 \in \mathcal{A}(M_\alpha)$ by Theorem~\ref{thm:1atomic}, the two sides of the equality $4\alpha = 2 \alpha^2 + 1$ yield distinct factorizations of the same element of $M_\alpha$ (see also Proposition~\ref{prop:HFM/UFM characterization} in the next section).
	\end{example}

	\bigskip
	\section{Factoriality and Elasticity}
	\label{sec:factoriality}
	
	In this last section, we characterize the monoids $M_\alpha$ that are half-factorial, and we briefly discuss the elasticity of $M_\alpha$. The elasticity is a factorization invariant that measures how far from being half-factorial a given monoid is.
	
	\smallskip
	\subsection{Half-Factoriality}
	
	Recall that an atomic monoid $M$ is an HFM if $|\mathsf{L}(x)| = 1$ for every $x \in M$. In the class consisting of evaluation monoids of Laurent semirings, being an HFM and being a UFM are equivalent conditions. We determine such monoids in the following proposition.
	
	\begin{prop} \label{prop:HFM/UFM characterization}
		For $\alpha \in \mathbb{R}_{>0}$, the following statements are equivalent.
		\begin{enumerate}
			\item[(a)] $M_\alpha$ is an UFM.
			\smallskip
			
			\item[(b)] $M_\alpha$ is an HFM.
			\smallskip
			
			\item[(c)] $\alpha = 1$ or $\alpha$ is transcendental. 
		\end{enumerate}
	\end{prop}
	
	\begin{proof}
		(a) $\Rightarrow$ (b): This follows by definition.
		\smallskip
		
		(b) $\Rightarrow$ (c): Suppose for the sake of contradiction that $\alpha$ is an algebraic number not equal to~$1$. Let $(p_\alpha(x), q_\alpha(x))$ be the minimal pair for $m_\alpha(x)$ over $\Z$. Because $M_\alpha$ is an HFM, it is an atomic monoid; thus, $\mathcal{A}(M_\alpha) = \{\alpha^n \mid n \in \Z\}$ by Theorem~\ref{thm:1atomic}. Hence, $z_p = p_\alpha(\alpha)$ and $z_q = q_\alpha(\alpha)$ are factorizations for the same element of $M_\alpha$. As $M_\alpha$ is an HFM, $p_\alpha(1) = |z_p| = |z_q| = q_\alpha(1)$, which implies that $1$ is a root of $m_\alpha(x)$. However, this contradicts the irreducibility of $m_\alpha(x)$.
		\smallskip
		
		
		(c) $\Rightarrow$ (a): If $\alpha=1$, then $M_\alpha = \nn_0$; hence, it is a UFM. On the other hand, suppose that $\alpha$ is transcendental. Then any equality of the form $1 = \sum_{n \in \Z} c_n \alpha^n$, where all but finitely many $c_n$ are zero, implies that $c_0 = 1$ and $c_n = 0$ for every $n \neq 0$. Therefore $1 \in \mathcal{A}(M_\alpha)$, and it follows from Theorem~\ref{thm:1atomic} that $M_\alpha$ is atomic. Now suppose that $p(\alpha)$ and $q(\alpha)$ are two factorizations of the same element in $M_\alpha$, where $p(x), q(x) \in \nn_0[x,x^{-1}]$. Take $k \in \nn$ such that $f(x) := x^k(p(x) - q(x)) \in \Z[x]$. Since $f(\alpha) = 0$, the fact that $\alpha$ is transcendental ensures that $f(x) = 0$ and, hence, $p(x) = q(x)$. Thus, the factorizations $p(\alpha)$ and $q(\alpha)$ are identical, concluding that $M_\alpha$ is a UFM.
	\end{proof}
	
	We proceed to discuss a dual notion of half-factoriality. A monoid $M$ is called a \emph{length-factorial monoid} (LFM) provided that for all $a \in M$ and $z,z' \in \mathsf{Z}(a)$, the equality $|z| = |z'|$ implies that $z = z'$. Observe that every UFM is an LFM. The notion of length-factoriality was first considered in~\cite{CS11} under the term ``other-half-factoriality," and it has been recently investigated in~\cite{CCGS21,GZ21,fG20}. On the other hand, not every LFM is a UFM, as illustrated next. 
	\smallskip
	
	\begin{example}
		Let $q \in \qq_{> 1} \setminus \nn$ and consider the additive submonoid $M$ of $\qq_{\ge 0}$ generated by the set $\{1,q\}$. Since $1 = \min M^\bullet$ and $q \notin \nn$, we conclude that $\mathcal{A}(M) = \{1,q\}$. In addition, one can check that if $z_1 := m_1 + n_1 q$ and $z_2 := m_2 + n_2q$ are two factorizations of the same element of $M$ having the same lengths, then $m_1 + n_1 = m_2 + n_2$ and, therefore, $(m_1, n_1) = (m_2, n_2)$; that is, $z_1 = z_2$. Thus, $M$ is an LFM. However, $M$ is not a UFM since, for instance, the two sides of the equality $\mathsf{n}(q) \cdot 1= \mathsf{d}(q) \cdot q$ yield distinct factorizations of $\mathsf{n}(q)$ in $M$. Additive submonoids of $\qq_{\ge 0}$ that are LFMs have been determined in \cite[Proposition~2.2]{fG20a}.
	\end{example}
	
	\begin{prop}
		For $\alpha \in \R_{>0},$ $M_\alpha$ is an LFM if and only if it is a UFM. 
	\end{prop}
	
	\begin{proof}
		If $\alpha$ is transcendental, then $M_\alpha$ is a UFM; hence, the statement of the proposition immediately follows. Then we assume that $\alpha$ is algebraic. It suffices to argue the direct implication, for the reverse implication follows by definition. To do this, suppose, by way of contradiction, that $M_\alpha$ is not a UFM. Then there exists an element of $M_\alpha$ having two distinct factorizations, namely, $p(\alpha)$ and $q(\alpha)$, where $p(x), q(x) \in \N_0[x,x^{-1}]$. After rearranging $(\alpha-1)p(\alpha) = (\alpha-1)q(\alpha)$, we obtain that $z_1 := \alpha p(\alpha)  + q(\alpha)$ and $z_2 := \alpha q(\alpha) + p(\alpha)$ are factorizations of the same element in $M_\alpha$. Observe that $z_1 \neq z_2$ as, otherwise, the Laurent polynomials $p(x)$ and $q(x)$ would satisfy $xp(x) + q(x) = xq(x) + p(x)$, which is not possible because $p(x) \neq q(x)$. However, the fact that $|z_1| = p(1) + q(1) = |z_2|$ indicates that $z_1$ and $z_2$ are distinct factorizations of the same element having the same length, which contradicts the fact that $M_\alpha$ is an LFM and completes the proof.
	\end{proof}
	
	Now we can summarize the main results we have established in this paper via the following diagram of implications, which is a specialization of Diagram~\eqref{eq:atomic chain} for the class consisting of all the evaluation monoids of Laurent semirings. As illustrated in Examples~\ref{ex:atomic monoids without the ACCP} and~\ref{ex:FFM that is not a UFM} the two (one-way) implications in the diagram are not reversible.

	\begin{equation} \label{eq:refined_chain}
		[\textbf{UFM} \ \Leftrightarrow \ \textbf{HFM} \ \Leftrightarrow \ \textbf{LFM}] \Rightarrow [\textbf{FFM} \ \Leftrightarrow \ \textbf{BFM} \ \Leftrightarrow \ \textbf{ACCP}] \ \Rightarrow \ \textbf{atomicity}
	\end{equation}

	\medskip
	\subsection{The Elasticity}
	
	We conclude this paper by saying a few words about the elasticity of the monoids $M_\alpha$. 
	Let $M$ be an atomic monoid. The \emph{elasticity} of a nonzero element $x \in M$, denoted by $\rho(x)$, is defined as
	\[
	\rho(x) := \frac{\sup \mathsf{L}(x)}{\min \mathsf{L}(x)}.
	\]
	In addition, we set $\rho(M) := \sup \{\rho(x) \mid x \in M^\bullet\}$ and call it the \emph{elasticity} of $M$. Notice that $\rho(M) \ge 1$. Furthermore, observe that $\rho(M) = 1$ if and only if $M$ is an HFM. As a result, the elasticity provides a measure of how far is an atomic monoid from being half-factorial.
	\smallskip
	
	As we proceed to argue, the elasticity of every monoid $M_\alpha$ is either $1$ or infinity.
	
	\begin{prop} \label{prop:elasticity}
		If $\alpha \in \rr_{> 0}$, then $\rho(M_\alpha) = 1$ if either $\alpha = 1$ or $\alpha$ is transcendental, and $\rho(M_\alpha) = \infty$ otherwise.
	\end{prop}
	
	\begin{proof}
		If $\alpha=1$ or $\alpha$ is transcendental, it follows from Proposition~\ref{prop:HFM/UFM characterization} that $M_\alpha$ is an HFM and, therefore, $\rho(M_\alpha) = 1$.
		\smallskip
		
		Now suppose that $\alpha$ is algebraic and $\alpha \neq 1$. We construct a sequence $(\beta_n)_{n \in \N}$ with terms in $M_\alpha$ such that $\sup \{\rho(\beta_n) \mid n \in \nn)\} = \infty$. Let $(p(x), q(x))$ be the minimal pair of $\alpha$. Then $z_1 := p(\alpha)$ and $z_2 := q(\alpha)$ are two distinct factorizations of the same element, namely, $\beta_1 \in M_\alpha$. Since $1$ is not a root of the minimal polynomial of $\alpha$, we see that $p(1) \neq q(1)$, so $z_1$ and $z_2$ are factorizations of different lengths. Suppose, without loss of generality, that $|z_1| < |z_2|$. For each $n \in \nn$, set $\beta_n = \beta_1^n.$ Then we see that, for every $n \in \nn$, both $z_1^n$ and $z_2^n$ are factorizations of $\beta_n$ in $M_\alpha$ whose lengths are $p(1)^n$ and $q(1)^n$, respectively. Therefore
		\[
		\rho(M_\alpha) \ge \rho(\beta_n) = \frac{\sup \mathsf{L}(\beta_n)}{\min \mathsf{L}(\beta_n)} \ge \frac{q(1)^n}{p(1)^n} =  \bigg( \frac{|z_2|}{|z_1|} \bigg)^n
		\]
		for every $n \in \nn$. Since $|z_2|/|z_1| > 1$, it follows that $\rho(M_\alpha) = \infty$, which concludes the proof.
	\end{proof}

	\bigskip
	\section*{Acknowledgments}
	
	First and foremost, it is my pleasure to thank Dr.~Felix Gotti for suggesting this project and for his mentorship all the way through. I also thank the MIT PRIMES-USA program for their support, without which this paper would not exist.

	\bigskip

\end{document}